\documentclass[12pt, a4paper, reqno, oneside]{amsart}
\usepackage{amsmath,amsthm,amssymb,mathrsfs}
\usepackage{color,hyperref,graphicx,comment}
\usepackage[utf8]{inputenc}
\usepackage[all]{xy}
\usepackage{mathtools}
\usepackage{enumitem}
\usepackage{faktor}
\usepackage{tikz-cd}
\usepackage[all]{xy} 
\usepackage{upgreek}
\usepackage{bbold} 
\usepackage{stmaryrd} 
\usepackage{cleveref}
\usepackage[
    block = ragged, 
    style = alphabetic, 
]{biblatex}

\allowdisplaybreaks

\usepackage{geometry}
\newcommand{\R}{\mathbb{R}}

\renewcommand{\phi}{\varphi}
\renewcommand{\epsilon}{\varepsilon}

\DeclarePairedDelimiter\ip{\langle}{\rangle} 

\newcommand{\integrald}{\textnormal{\slshape d}}

\newcommand{\I}[3][)]{\if)#1 \int #2 \, \integrald#3
    \else\int_{#1} #2 \, \integrald#3\fi}

\DeclareMathOperator{\tr}{tr}

\DeclareMathOperator{\diver}{div}

\newcommand{\grad}{\nabla}
\DeclareMathOperator{\Hess}{Hess}

\newcommand{\dvol}{dV_{g}}
\DeclareMathOperator{\Ric}{Ric}

\numberwithin{equation}{section}

\theoremstyle{plain}
  \newtheorem{theorem}{Theorem}[section]
  \newtheorem{lemma}[theorem]{Lemma}
  \newtheorem{proposition}[theorem]{Proposition}
  \newtheorem{corollary}[theorem]{Corollary}
  \newtheorem*{theorem*}{Theorem}  
  \newtheorem*{theoremm*}{Main Theorem}
\theoremstyle{definition}
  \newtheorem{definition}[theorem]{Definition}
  \newtheorem*{definition*}{Definition}
  
  \newtheorem{remark}[theorem]{Remark}
  
  \newtheorem*{ack*}{Acknowledgements}
  \newtheorem*{ref*}{Reference}
  
  \newtheorem*{ft*}{Fact}
\theoremstyle{plain}

\title{Classification of Steady Gradient Ricci-Yang-Mills Solitons on Surfaces}
\author{Michael Womack}
\thanks{The author is supported by the NSF via DMS-2342135. The author would like to thank his advisor Jeffrey Streets for suggesting the problem and his patience and guidance along the way.}

\begin{document}

\begin{abstract}
  We construct string backgrounds in dimension 2 which connect the Hamilton cigar to the round sphere. Specifically, we construct a 1-parameter family of rotationally symmetric steady gradient Ricci-Yang-Mills solitons on surfaces, where we denote the parameter by $\lambda\in[-2,\infty)$. At $\lambda=-2$ is the Hamilton cigar, for $-2<\lambda<0$ the solitons are asymptotic to cylinders, at $\lambda=0$ is a complete noncompact soliton forming a cusp at infinity, and as $\lambda$ approaches infinity the family approaches a round point. Furthermore, we show any complete steady gradient Ricci-Yang-Mills soliton on a surface must come from this family.
\end{abstract}

\maketitle

\section{Introduction}
Let $(M,g)$ be an oriented Riemannian manifold, $H$ a closed 3-form on $M$ and $f\in C^{\infty}(M)$. We say $(M,g,H,f)$ is a \emph{steady generalized Ricci soliton} if it satisfies the system of equations
\begin{equation*}
  \begin{split}
    \Ric-\tfrac{1}{4}H^{2}+\Hess{f} &= 0,\\
    d^*(e^{-f}H) &= 0.
  \end{split}
\end{equation*}
Generalized Ricci solitons play a role in the study of the generalized Ricci flow that is analogous to the role of classical Ricci solitons in the Ricci flow. The generalized Ricci flow is a natural coupling of the Ricci flow to a heat flow for a certain closed 3-form. Notably, the generalized Ricci flow is the gradient flow of the lowest eigenvalue of a certain Schr\"odinger operator and satisfies a Perelman-type entropy monotonicity formula, so solitons are expected to model the singularities and long time limits of the flow \cite{mgf_streetsGRF2021}. The generalized Ricci flow equation originally appeared in the physics literature as the renormalization group flow of a nonlinear sigma model \cite{polchinski_StringTheoryVol11998}. 

Recently, generalized Ricci flow has been investigated as a potential tool for constructing and analyzing solutions to the Hull-Strominger system from string theory via the geometry of string Courant algebroids in generalized geometry. Generalized Ricci flow (and various related geometric flows) satisfy a dimension reduction principle which allows solutions of one flow to be related to corresponding solutions in the other flows. Refer to \cite{garciafernandezMolinaStreetsPluriclosedFlowHullStrominger2026} for further details. 

When the dimension of the base manifold is 2, the reduced PDE is the Ricci-Yang-Mills flow. As the name suggests, the Ricci-Yang-Mills flow is a coupling of the Ricci flow with the Yang-Mills flow. This geometric evolution equation was originally introduced to attempt to find special metrics on manifolds in a similar manner to the Ricci flow, but with the hope that the presence of a bundle connection would simplify the analysis \cite{streetsRYMFlow2007}. 

Let $(M,g)$ be a Riemannian manifold and $L\to M$ be the total space of a $U(1)$-bundle over $M$. Denote by $A$ a connection on this bundle with curvature 2-form $F$ and let $f\in C^{\infty}(M)$. The tuple $(M,g,F,f)$ is said to be a \emph{steady gradient Ricci-Yang-Mills soliton} if it satisfies the following system:
\begin{equation*}
  \begin{split}
    \Ric(g)-\frac{1}{2}F^2+\Hess f &= 0,\\
    d^{*}(e^{-f}F) &= 0,
  \end{split}
\end{equation*}
where $F_{ij}^{2}=g^{jl}F_{ij}F_{kl}$. The Ricci-Yang-Mills flow has a similar analytical theory as the generalized Ricci flow mentioned above. In the context of the dimension reduction principle, a Ricci-Yang-Mills soliton on a surface can be used to construct a generalized Ricci soliton on a 3-manifold, which can in turn be used to produce a pluriclosed soliton on a complex surface. Examples of this procedure for the generalized Ricci soliton to pluriclosed soliton construction are done in section 4 of \cite{streetsClassSolitonsPluriclosedFlow2019}.

In dimension 3 Bryant constructed a rotationally symmetric, positively curved noncompact steady gradient Ricci soliton and proved that it is the \textit{unique} soliton satisfying these properties \cite{bryantRFSol3DRotSym2005}. In contrast to the classical setting, Podest\'a and Raffero constructed a 1-parameter family of complete, positively curved, noncompact, rotationally symmetric generalized Ricci solitons in dimension 3 which are pairwise non-isometric \cite{podestaRafferoPosCurvSO3GenRS2025}. For generalized Ricci flow, when the 3-form is zero the equations reduce to the Ricci flow, and in the Podest\'a-Raffero construction this yields the Bryant soliton in their family. Naturally this motivates the question as to whether this phenomena occurs in other settings for extensions of the Ricci flow.

In dimension 2 there is a unique positively curved noncompact steady gradient Ricci soliton called the Hamilton cigar (see Lemma 4.8 of \cite{chowLuNi_HamiltonsRicciFlow} or Theorem 3.11 of \cite{chow_RicciSolLowDim2023}). In this paper we construct a 1-parameter family of complete, noncompact rotationally symmetric Ricci-Yang-Mills solitons on surfaces connecting the Hamilton cigar soliton to a soliton which develops a cusp at infinity. The family of solitons includes compact solitons on orbifolds, which under an appropriate rescaling connect to the classical shrinking sphere. This family of solitons roughly follows this schematic:
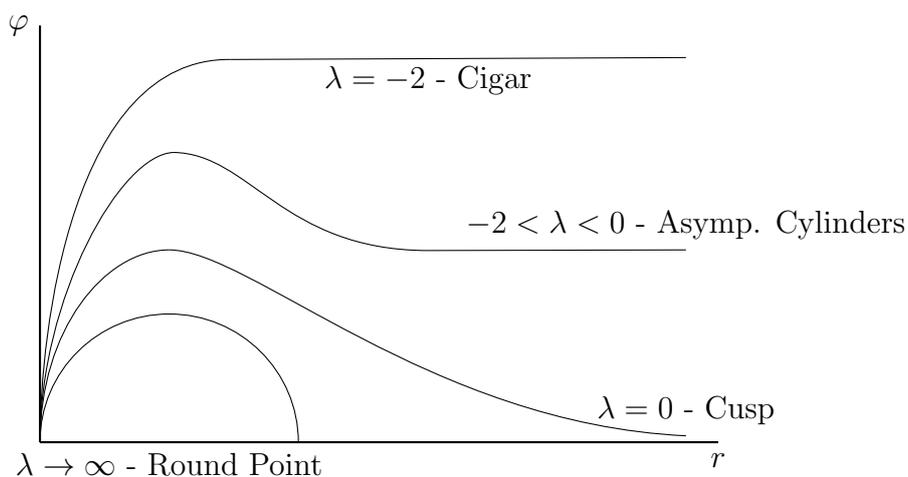
\begin{figure}[h!]
  \begin{tikzpicture}[scale=0.85]
  \coordinate (A) at (0.00, 0.00);
  \coordinate (B) at (2.95, 5.97);
  \coordinate (C) at (0.00, 1.75);
  \coordinate (D) at (0.50, 6.00);
  \coordinate (E) at (10.00, 6.00);
  \coordinate (F) at (2.07, 4.52);
  \coordinate (G) at (5.99, 2.99);
  \coordinate (H) at (2.00, 3.00);
  \coordinate (I) at (10.00, 0.10);
  \coordinate (J) at (0.00, 2.25);
  \coordinate (K) at (1.25, 4.50);
  \coordinate (L) at (3.25, 4.50);
  \coordinate (M) at (3.75, 3.00);
  \coordinate (N) at (10.00, 3.00);
  \coordinate (O) at (0.00, 1.75);
  \coordinate (P) at (1.00, 3.00);
  \coordinate (Q) at (6.03, 0.29);
  \coordinate (R) at (3.25, 3.00);
  \coordinate (S) at (2.00, 0.00);
  \coordinate (T) at (6.01, 6.04);
  \coordinate (G_1) at (5.86, 3.07);
  \coordinate (U) at (5.97, 1.39);
  \coordinate (V) at (1.99, 1.28);
  \coordinate (W) at (10.50, 0.00);
  \coordinate (X) at (0.00, 6.50);

  \draw (A) .. controls (0.00, 1.75) and (0.50, 6.00) .. (B);
  \draw (B) -- (E);
  \draw (A) .. controls (0.00, 2.25) and (1.25, 4.50) .. (F);
  \draw (F) .. controls (3.25, 4.50) and (3.75, 3.00) .. (G);
  \draw (G) -- (N);
  \draw (A) .. controls (0.00, 1.75) and (1.00, 3.00) .. (H);
  \draw (I) .. controls (6.03, 0.29) and (3.25, 3.00) .. (H);
  \draw (4.00, 0.00) arc (0.0:180.0:2.00);
  \node[below] at (T) {$\lambda = -2$ - Cigar};
  \node[above] at (N) {$-2<\lambda<0$ - Asymp. Cylinders};
  \node[above] at (I) {$\lambda=0$ - Cusp};
  \node[below] at (S) {$\lambda\to\infty$ - Round Point};
  \draw[thick] (A) -- (W);
  \draw[thick] (X) -- (A);
  \node[below] at (W) {$r$};
  \node[left] at (X) {$\varphi$};
  \end{tikzpicture}
  \caption{The warped product profile function $\phi(r)$ for different values of $\lambda$.}
\end{figure}

More precisely, we have (cf. Theorem \ref{convergencethm}):
\begin{theorem}\label{mainthm1}
There is a 1-parameter family of rotationally symmetric steady gradient Ricci-Yang-Mills solitons on surfaces, where $\lambda\in [-2,\infty)$ parametrizes the family. We have the following: The family converges
  \begin{enumerate}[label=(\arabic*)]
    \item in the pointed Cheeger-Gromov sense to the Hamilton cigar soliton as $\lambda\to -2^{+}$;
    \item in the pointed Cheeger-Gromov sense to the cusp soliton as $\lambda\to 0^{-}$;
    \item in the pointed Gromov-Hausdorff sense to the cusp soliton as $\lambda\to 0^{+}$;
    \item in the Gromov-Hausdorff sense, after a rescaling, to $S^{2}$ with the round metric as $\lambda\to \infty$.
  \end{enumerate}
\end{theorem}

It is well known that in the surface case for Ricci solitons, the curvature must have a sign (cf. \cite{chow_RicciSolLowDim2023}). This curvature rigidity no longer holds in the Ricci-Yang-Mills soliton case; instead, a natural curvature associated to Ricci-Yang-Mills solitons (corresponding to $\lambda$) will have a sign (see Lemma \ref{potentCurvIdent}). We expect that the orbifold solitons described in Theorem \ref{mainthm1} to be the bases of the solitons constructed in \cite{streetsUstinovskiy_ClassificationGenKRSolitonCmplxSurf2020}.

\begin{remark}
  For the noncompact solutions, the Yang-Mills density $|F|^2$ decays to zero as $r$ approaches infinity. The cusp soliton has the slowest rate of decay, with $|F|^2\sim \frac{1}{r^4}$.
\end{remark}

Furthermore, we show that this construction is exhaustive (cf. Theorem \ref{classification_thm}):
\begin{theorem}\label{mainthm2}
  Any complete steady gradient Ricci-Yang-Mills soliton on a surface must be rotationally symmetric. Therefore, it must be a member of the family described in Theorem \ref{mainthm1}.
\end{theorem}

\begin{remark}
  The classical steady Ricci solitons on surfaces include the cylinder and flat torus solutions. Theorem \ref{mainthm1} classifies the solutions with nonzero bundle curvature, which rules out the cylinder and flat torus solutions.
\end{remark}

The outline of this paper is as follows. In \S 2 we provide some background on Ricci-Yang-Mills solitons and derive relevant identities in the surface case. In \S 3 we reduce the system of equations for the potential function {}$f$ and warping factor in the rotational symmetry construction to ODEs for each, and then prove existence of solutions for the different regimes of $\lambda$ for Theorem \ref{mainthm1}. In \S 4 we prove the various convergence results to finish Theorem \ref{mainthm1}. In \S 5 we show that the ODEs of \S 3 can be derived intrinsically from the Ricci-Yang-Mills soliton equations, and assuming completeness the solitons must be rotationally symmetric, yielding Theorem \ref{mainthm2}.

\section{Ricci-Yang-Mills Solitons}
We begin this section with a review of Ricci-Yang-Mills solitons and identities necessary for deriving ODEs for the potential function and metric warping factor in the surface case.
\begin{definition}\label{sgRYMSoliton_def}
  Let $(M,g)$ be a Riemannian manifold, $f\in C^{\infty}(M)$ and $F$ a two form on $M$. We say that $(M,g,f,F)$ is a \emph{steady gradient Ricci-Yang-Mills soliton} if it satisfies the following system:
  \begin{align*}
      \Ric(g)-\frac{1}{2}F^2+\Hess f &= 0\\
      d^{*}(e^{-f}F) &= 0.
  \end{align*}
\end{definition}

It is well known in the surface case that for a nontrivial gradient Ricci soliton, $J(\nabla f)$ is a nontrivial Killing field (in fact this is true for gradient K\"ahler-Ricci solitons generally). The next proposition shows this holds for Ricci-Yang-Mills solitons on surfaces. 
\begin{proposition}[Proposition 17 of \cite{streets_RicYMFlowSurfaces2009}]
  Let $(M,g,f,F)$ be a nontrivial Ricci-Yang-Mills soliton on a surface with $J$ the standard complex structure. Then $J(\nabla f)$ is a nontrivial Killing field.
\end{proposition}
\begin{proof}
  For a Ricci-Yang-Mills soliton, we have that
  \[
    \mathcal{L}_{\nabla f} g = 2\Hess f = \lambda g-\Ric+\frac{1}{2}F^{2}
  \]
  for a constant $\lambda$. We can take $\lambda \in \{-\tfrac{1}{2},0,\tfrac{1}{2}\}$ in which cases we call the solitons shrinking, steady or expanding respectively. On a surface we have $\Ric = \tfrac{1}{2}R(g)g$, where $R(g)$ is the scalar curvature of $g$, and $F^2=\tfrac{1}{2}|F|^{2}g$. We compute in normal coordinates that
  \begin{align*}
    (\mathcal{L}_{J(\nabla f)}g)_{ij} &= \nabla_{i}\left(J_{j}^{k}\nabla_{k}f\right)+\nabla_{j}\left(J_{i}^{k}\nabla_{k}f\right)\\
      &= \left(J_{j}^{k}\nabla_{i}\nabla_{k}f\right)+\left(J_{i}^{k}\nabla_{j}\nabla_{k}f\right)\\
      &= \tfrac{1}{2}\left(J_{j}^{k}(\lambda-\tfrac{1}{2}R(g)+\tfrac{1}{4}|F|^{2})g_{ik})\right)+\tfrac{1}{2}\left(J_{i}^{k}(\lambda-\tfrac{1}{2}R(g)+\tfrac{1}{4}|F|^{2})g_{jk})\right)\\
      &= \tfrac{1}{2}(\lambda-\tfrac{1}{2}R(g)+\tfrac{1}{4}|F|^{2})\left(J_{j}^{k}g_{ik}+J_{i}^{k}g_{jk}\right)\\
      &= 0
  \end{align*}
  where the last line follows since $J_{j}^{k}g_{ik}+J_{i}^{k}g_{jk}=0$.
\end{proof}

Assuming that the Killing field has a fixed point guarantees rotational symmetry.
\begin{lemma}[Lemma 1 of \cite{chen_NoteUniformRiemSurf2006}]\label{fix_point_rot_sym_lemma}
  Let $(M^{2},g)$ be a two dimensional complete Riemannian manifold with non trivial Killing vector field $X$. Suppose $X$ vanishes at $O\in M$, then $(M^{2},g)$ is rotationally symmetric.
\end{lemma}

We now assume that the potential function $f$ has a critical point, which implies the Killing field $J(\nabla f)$ has a fixed point. Under this ansatz, we use the rotational symmetry of the metric to set up a system of PDEs in the warping function $\phi$ and potential function $f$. Let $r$ denote the radial arc length parameter and $\theta$ the angular coordinate in polar coordinates, with the origin taken to be a point where the Killing field $J(\nabla f)$ vanishes.

\begin{lemma}[Conservation Law]\label{YMSurfSolConserve1}
  Let $(M^{2}, g, F, f)$ be a Ricci-Yang-Mills soliton. Suppose $g=dr^{2}+\phi^{2}(r)d\theta^{2}$ and $F=\psi(r) dr\wedge d\theta$. The quantity
  \[
    \frac{\psi}{\phi}e^{-f} = \eta
  \]
  is constant.
\end{lemma}
\begin{proof}
  First note that $\dvol = \phi dr\wedge d\theta$. We have that $*\psi dr\wedge d\theta = \frac{\psi}{\phi}$, where $*$ denotes the Hodge star operator. From the Yang-Mills component of the soliton equation, we compute
  \[
    0 = d^{*}\left(e^{-f}F\right) = *d*\left(e^{-f}F\right)=*d\left(\frac{\psi}{\phi}e^{-f}\right)
  \]
  so $\frac{\psi}{\phi}e^{-f} = \eta$ for some $\eta\in\R$.
\end{proof}

The above lemma lets us write that $F=\eta e^{f}\dvol$, which will be important for deriving the system of equations for $f$ and $\varphi$.

\begin{lemma}[$F^2$ identities]\label{Fsq_comp}
  Suppose $g=dr^{2}+\phi^{2}(r)d\theta^{2}$ and $F=\eta e^{f}\dvol=\psi(r) dr\wedge d\theta$. Then
  \begin{align*}
    F^{2} &= e^{2f}\eta^{2}dr^{2}+e^{2f}\eta^{2}\phi^{2}d\theta^{2} = \frac{\psi^{2}}{\phi^{2}}dr^{2}+\psi^{2}d\theta^{2} = \eta^{2}e^{2f}g,\\
    |F|^{2} &= 2\eta^{2} e^{2f} = 2\frac{\psi^{2}}{\phi^{2}}.
  \end{align*}
\end{lemma}
\begin{proof}
  First note that $F=\eta e^{f}\dvol=\eta e^{f}\varphi \;dr\wedge d\theta$, and that 
  \[
    g = \begin{pmatrix}
      1 & 0\\
      0 & \phi^{2}(r)
    \end{pmatrix} \implies g^{-1} = \begin{pmatrix}
      1 & 0\\
      0 & \phi^{-2}(r)
    \end{pmatrix}.
  \]
  Since $F_{ij}^{2} = g^{kl}F_{ik}F_{jl}$, the only components which are nonzero are those corresponding to $i=j=r$ and $i=j=\theta$. Thus we have
  \begin{align*}
    F_{rr}^{2} &= g^{\theta\theta}F_{r\theta}F_{r\theta} e^{2f}= \phi^{-2}\eta^{2}\phi^{2}e^{2f}= e^{2f}\eta^{2},\\
    F_{\theta\theta}^{2} &= g^{rr}F_{\theta r}F_{\theta r} = \eta^{2}\phi^{2}e^{2f},
  \end{align*}
  so that
  \[
    F^{2} = e^{2f}\eta^{2}dr^{2}+e^{2f}\eta^{2}\phi^{2}d\theta^{2} = \eta^2 e^{2f}g.
  \]
  The case that $F=\psi dr\wedge d\theta = \frac{\psi}{\phi}\dvol$ follows easily.

  To compute $|F|^{2}$, observe that $\ip{dr,dr}=1$, $\ip{dr,d\theta}=0$, $\ip{d\theta,d\theta}=\frac{1}{\phi^{2}}$, so 
  \[
    |F|^{2}= 2\psi^{2}\ip{dr,dr}\ip{d\theta,d\theta}= 2\frac{\psi^{2}}{\phi^{2}}.
  \]
\end{proof}

\begin{lemma}[System for $\phi$, $\psi$ and $f$]\label{YMSurfSolMetricEq}
  Let $(M^{2}, g, F, f)$ be a Ricci-Yang-Mills soliton. Suppose $g=dr^{2}+\phi^{2}(r)d\theta^{2}$ and $F=\psi dr\wedge d\theta$. The metric part of the soliton equation yields the two equations
  \begin{align*}
    0 &= -\frac{\phi''}{\phi}+f''-\frac{1}{2}\frac{\psi^{2}}{\phi^{2}},    \\
    0 &= -\phi\phi''+\phi\phi'f'-\frac{1}{2}\psi^{2},
  \end{align*}
  which together imply $f'=\mu\phi$ for some $\mu\in\R$. The Yang-Mills component of the soliton equation yields
  \[
    \frac{\psi'\phi-\psi\phi'}{\phi} = \psi f'.
  \]
\end{lemma}
\begin{proof}
  This follows from formulas for warped product metrics. Given $g=dr^{2}+\phi^{2}(r)d\theta^{2}$, then
  \begin{align*}
    \Ric(g) &= -\frac{\phi''}{\phi} dr^{2} -(\phi\phi'')d\theta^{2}=-\frac{\phi''}{\phi}g,\\
    \Hess f &= f''dr^{2}+\phi\phi'f'd\theta^{2},\\
    F^2 &= \frac{\psi^{2}}{\phi^{2}}dr^{2}+\psi^{2}d\theta^{2},
  \end{align*}
  where the first two lines are from \cite{chow_etalRicciFlowVol1_2007} Lemma 1.21, and the third line is Lemma \ref{Fsq_comp}.

  Substituting these into the symmetric part of the Ricci-Yang-Mills soliton equation yields
  \begin{align*}
    0 &= -\frac{\phi''}{\phi}+f''-\frac{1}{2}\frac{\psi^{2}}{\phi^{2}},    \tag*{($dr^{2}$ part)}\\
    0 &= -\phi\phi''+\phi\phi'f'-\frac{1}{2}\psi^{2},    \tag*{($d\theta^{2}$ part)}
  \end{align*}
  which together yield that
  \[
    f'' = \frac{\phi' f'}{\phi} \implies f' = \mu\phi 
  \]
  for some $\mu\in\R$.

  For the Yang-Mills component of the soliton equation, observe that
  \[
    i_{\grad f}F = i_{\frac{\partial f}{\partial r}\frac{\partial}{\partial r}}\psi dr\wedge d\theta = \psi f' d\theta
  \]
  and 
  \[
    d^{*}F = *d*(\psi dr\wedge d\theta) = *d\left(\frac{\psi}{\phi}\right) = *\left(\frac{\psi'\phi-\psi\phi'}{\phi^{2}} dr\right) = \frac{\psi'\phi-\psi\phi'}{\phi} d\theta.
  \]
  Combining these equations yields the desired result.
\end{proof}

The next lemma is simply collecting a set of computations which will be used for computing important curvature identities in the following proposition. We note that in this lemma and the following proposition, we do not use any properties specific to the surface case, so these identities hold generally.

\begin{lemma}\label{diverFsq_identity}
  In coordinates,
  \[
    (\diver F^{2})_{j} = (d^{*}F)_{k}F_{jk}+\frac{1}{4}\nabla_{j}|F|^{2} = F^{2}_{kj}\nabla_{k}f+\frac{1}{4}\nabla_{j}|F|^{2}.
  \]
\end{lemma}
\begin{proof}
  We compute 
  \begin{align}
    (\diver F^{2})_{j} &= \nabla_{i} F_{ij}^{2} \nonumber\\
      &= \nabla_{i}(F_{ik}F_{jk}) \nonumber\\
      &= (\nabla_{i}F_{ik})F_{jk}+F_{ik}\nabla_{i}F_{jk} \label{eq:diver_form_1}\\
      &= (d^{*}F)_{k}F_{jk}+F_{ik}\left(\nabla_{j}F_{ik}-\nabla_{k}F_{ij}\right) \\
      &= (d^{*}F)_{k}F_{jk}+\frac{1}{2}\nabla_{j}|F|^{2}-F_{ik}\nabla_{k}F_{ij} \label{eq:diver_form_2},
  \end{align}
  where we used $dF=0$ in the fourth line.
  Setting $\eqref{eq:diver_form_1}=\eqref{eq:diver_form_2}$ yields that $\frac{1}{2}\nabla_{j}|F|^{2}=2F_{ik}\nabla_{i}F_{jk}$, so we get that
  \[
    (\diver F^{2})_{j} = (d^{*}F)_{k}F_{jk}+\frac{1}{4}\nabla_{j}|F|^{2}.
  \]
  The first term can be further rewritten using $\left(d^{*}F\right)_{i}=(i_{\nabla f}F)_{i} = F_{ij}\nabla_{i}f$ from the soliton equation, and thus  $(d^{*}F)_{j}F_{ij} = F_{kj}\nabla_{k}f\;F_{ij} = F^{2}_{ki}\nabla_{k}f$.
\end{proof}

\begin{proposition}[Constant Curvature Identities]\label{constCurvIdent}
  For a gradient Ricci-Yang-Mills soliton, we have the following identities:
  \begin{align}
    0 &= R-\frac{1}{2}|F|^{2}+\Delta f, \label{eq:constCurvIdent1}\\
    0 &= \nabla\left(R-\frac{3}{4}|F|^{2}+|\nabla f|^{2}\right), \label{eq:constCurvIdent2}\\
    0 &= \nabla\left(\frac{1}{4}|F|^{2}+\Delta f-|\nabla f|^{2}\right). \label{eq:constCurvIdent3}
  \end{align}
\end{proposition}
\begin{proof}
  Tracing $\Ric -\frac{1}{2}F^{2}+\nabla_{i}\nabla_{j}f=0$ yields that $R-\frac{1}{2}|F|^{2}+\Delta f=0$. We differentiate to get
  \[
    \nabla_{i} R = \frac{1}{2}\nabla_{i}|F|^{2}-\nabla_{i}\Delta f.
  \]
  Recall that we commute derivatives on $\nabla_{j}f$ with the following formula
  \[
    \nabla_{i}\nabla_{j}\nabla_{j} f = \nabla_{j}\nabla_{i}\nabla_{j} f- \Ric_{ij}\nabla_{j}f.
  \]
  We now compute
  \begin{align}
    \nabla_{i} R &= \frac{1}{2}\nabla_{i}|F|^{2}-\nabla_{i}\Delta f \nonumber\\
      &= \frac{1}{2}\nabla_{i}|F|^{2}-\left(\nabla_{j}\nabla_{i}\nabla_{j}f-\Ric_{ij}\nabla_{j}f\right) \nonumber\\
      &= \frac{1}{2}\nabla_{i}|F|^{2}-\left(\nabla_{j}\left(\frac{1}{2}F^{2}-\Ric\right)_{ij}\right)+\Ric_{ij}\nabla_{j}f \label{eq:hess_term_replace}\\
      &= \frac{1}{2}\nabla_{i}|F|^{2}-\frac{1}{2}\left(\diver F^{2}\right)_{i}+\frac{1}{2}\nabla_{i} R+\Ric_{ij}\nabla_{j} f \nonumber 
  \end{align}
  where in \eqref{eq:hess_term_replace} we used the soliton equation to replace the Hessian term. Subtracting the $\nabla_{i}R$ term and multiplying by 2 yields
  \[
    \nabla_{i}R = \nabla_{i}|F|^{2}-(\diver F^{2})_{i}+2\Ric_{ij}\nabla_{j} f.
  \]
  We then apply Lemma \ref{diverFsq_identity} to replace the divergence term and compute
  \begin{align*}
    \nabla_{i}R &= \nabla_{i}|F|^{2}-\frac{1}{4}\nabla_{i}|F|^{2}-F_{ji}^{2}\nabla_{j}f+2\Ric_{ij}\nabla_{j}f\\
      &= \frac{3}{4}\nabla_{i}|F|^{2}-F_{ji}^{2}\nabla_{j} f+2\Ric_{ij}\nabla_{j}f\\
      &= \frac{3}{4}\nabla_{i}|F|^{2}+\left(2\Ric_{ij}-F_{ij}^{2}\right)\nabla_{j}f.
  \end{align*}
  Finally, since $(2\Ric_{ij}-F_{ij}^{2})=(-2\nabla_{i}\nabla_{j}f)\nabla_{j}f=-\nabla_{i}|\nabla f|^{2}$, we have that
  \[
    \nabla_{i}\left(R-\frac{3}{4}|F|^{2}+|\nabla f|^{2}\right) = 0.
  \]

  To prove \eqref{eq:constCurvIdent3}, differentiate \eqref{eq:constCurvIdent1} and subtract \eqref{eq:constCurvIdent2}.
\end{proof}

\begin{lemma}[Potential-Curvature Identity]\label{potentCurvIdent}
  Suppose $g=dr^{2}+\phi^{2}(r)d\theta^{2}$. Then the final identity in Proposition \ref{constCurvIdent} yields the two equivalent equations
  \begin{align*}
    \frac{1}{2}\eta^2 e^{2f}-(f')^{2}+(f''+(\ln(\phi))'f') &= \lambda\\
    \frac{1}{2}\frac{\psi^2}{\phi^2}-\mu^{2}\phi^{2}+2\mu\phi' &= \lambda.
  \end{align*}
\end{lemma}
\begin{proof}
  Note that the left side of each equation is the coordinate expression for $\frac{1}{4}|F|^2$ with the corresponding functions in the ODE. This is the same setup as Lemma \ref{YMSurfSolMetricEq}, so we recall that 
  \[
    \Hess{f} = f''(r)dr^{2}+\phi\phi'f'd\theta^{2} \implies \Delta f = \tr_{g}\Hess{f} = f''(r)+\frac{\phi'}{\phi}f' = f''(r)+(\ln(\phi))'f'.
  \]
  From \eqref{eq:constCurvIdent3} we have
  \begin{equation}\label{Lambda_definition}
    \frac{1}{4}|F|^{2}+\Delta f-|\nabla f|^{2}=\lambda
  \end{equation}
  for some $\lambda\in\R$, and substituting $|\nabla f|^{2}=(f')^{2}$ into the above formula for $\Delta f$ yields the first desired identity. The second identity follows from $f'=\mu\phi$ easily.
\end{proof}

\begin{remark}
  On a compact manifold, the constant $\lambda$ is nonnegative. This can be easily seen by integrating \eqref{Lambda_definition} against $e^{-f}$ and integrating by parts.
\end{remark}

\section{ODE Reduction and Soliton Construction}
In this section we use the rotational symmetry ansatz to reduce the Ricci-Yang-Mills soliton equations to ODEs in the soliton potential and metric warping factor. We discuss a canonical rescaling to fix certain degrees of freedom for the ODEs, then construct solutions via elementary ODE analysis.

\subsection{ODE Reduction}
The rotational symmetry ansatz reduces the soliton equations to a system of ODEs as follows:
\begin{lemma}[Potential/Metric ODEs]\label{poten_metric_ODE}
  Assuming $\mu\ne 0$, the system in Lemma \ref{YMSurfSolMetricEq} satisfies the following equations for the potential $f$ and metric warping function $\phi$:
  \begin{align*}
    0 &=f^{(3)}-\frac{1}{2}\left((f')^{2}\right)'+\frac{\eta^{2}}{4}\left(e^{2f}\right)',\\
    0 &=\phi''-\frac{3\mu}{2}\left(\phi^{2}\right)'+\mu^{2}\phi^{3}+\lambda\phi.
  \end{align*}
\end{lemma}
\begin{proof}
  Starting with
  \[
    0 = -\frac{\phi''}{\phi}+f''-\frac{1}{2}\frac{\psi^{2}}{\phi^{2}}\implies 0 = -\phi''+\phi f''-\frac{1}{2}\phi\frac{\psi^{2}}{\phi^{2}}
  \]
  where we multiplied by $\phi$, apply $\phi = \frac{1}{\mu}f'$ to yield
  \begin{align*}
    0 &= -\frac{1}{\mu}f^{(3)}+\frac{1}{2\mu}\left((f')^{2}\right)'-\frac{\eta^{2}}{2\mu}f'e^{2f}\\
      &=f^{(3)}-\frac{1}{2}\left((f')^{2}\right)'+\frac{\eta^{2}}{4}\left(e^{2f}\right)'
  \end{align*}
  where we replaced the $\frac{\psi^{2}}{\phi^{2}}$ term using Lemma \ref{YMSurfSolConserve1}.

  Alternatively, we can write this as an ODE for $\phi$ using $f'=\mu\phi$,
  \begin{align*}
    0 &= -\phi''+\phi (\mu\phi')-\frac{1}{2}\phi\frac{\psi^{2}}{\phi^{2}}\\
      &= -\phi''+\phi (\mu\phi')-\phi \left(\mu^{2}\phi^{2}-2\mu\phi'+\lambda\right) \tag*{(Lemma \ref{potentCurvIdent})}\\
      &= -\phi''+\frac{\mu}{2}\left(\phi^{2}\right)'-\mu^{2}\phi^{3}+\mu(\phi^{2})'-\lambda\phi \\
      &= -\phi''+\frac{3\mu}{2}\left(\phi^{2}\right)'-\mu^{2}\phi^{3}-\lambda\phi.
  \end{align*}
\end{proof}

\begin{remark}
  We note here that the ODE for $f$ in Lemma \ref{poten_metric_ODE} derived from the metric part of the system of equations in Lemma \ref{YMSurfSolMetricEq} is equivalent to the ODE for $f$ derived from the constant curvature identities in Lemma \ref{potentCurvIdent}. The identity $f'=\mu\varphi$, which is needed to give that these two ODEs are equivalent for $f$, only comes from the metric part of Lemma \ref{YMSurfSolMetricEq} however. 
\end{remark}

It will be more convenient to work with the second order ODE for $f$, so we include it as an explicit corollary.
\begin{corollary}
  The potential function $f$ satisfies the 2nd order ODE
  \begin{equation}\label{f_2ndorder_eqt}
    f''-\frac{1}{2}(f')^{2}+\frac{\eta^2}{4} e^{2f} = \frac{1}{2}\lambda
  \end{equation}
  where $\lambda$ is the constant from Lemma \ref{potentCurvIdent}.
\end{corollary}

\begin{remark}
  The quantity $\lambda$ from Corollary \ref{f_2ndorder_eqt} is easily computed at $r=0$ to be
  \[
    \lambda=2\left(\frac{\eta^{2}}{4}+\mu\right).
  \]
\end{remark}

\subsection{Rescaling}
The rotational symmetry ansatz for steady gradient Ricci-Yang-Mills solitons on surfaces contains three degrees of freedom for specifying solutions. The first degree of freedom is that $f$ may be shifted by a constant (as the RYM soliton equations are satisfied with $f+c$). We fix this degree of freedom by assuming $f(0)=0$. The second degree of freedom is the parameter $\mu$ that appears in the identity $f'=\mu\phi$; the third degree of freedom is the choice of $\eta$ specifying the ``strength'' of the $F$ field from Lemma \ref{YMSurfSolConserve1}. Note that $\mu\in\R$, while $\eta>0$. These two parameters come together in the constant $\lambda$ from the constant curvature identities from Lemma \ref{potentCurvIdent}.

We can fix one of these two degrees of freedom by scaling the soliton, so we require the following lemma:
\begin{lemma}[Rescaling Lemma]\label{rescaling_lemma}
  If $(M,g,F,f)$ is a steady gradient Ricci-Yang-Mills soliton, then for $\alpha\in\R_{> 0}$ we have that $(M,\alpha g, \sqrt{\alpha}F, f)$ is also a steady gradient Ricci-Yang-Mills soliton.
\end{lemma}
\begin{proof}
  Note that $\Ric(g)$ and $\Hess f$ are invariant under scaling $g$, and 
  \[
    F^{2}_{(\alpha g)}(X,Y) = (\alpha g)(i_{X}F,i_{Y}F) = \frac{1}{\alpha}(g(i_{X}F,i_{Y}F)),
  \]
  so scaling $g$ by $\alpha$ and $F$ by $\sqrt{\alpha}$ preserves the first equation. The second equation is preserved under scalings as well, which proves the claim.
\end{proof}

Observe that taking $\mu=0$ in the ODE for $\phi$ in Lemma \ref{poten_metric_ODE} gives $\phi(r) = \tfrac{1}{\lambda}\sin(\lambda r)$, yielding the sphere solution. Since we are fixing $\mu=-1$, this solution now appears as $\lambda\to\infty$. 

\begin{lemma}[ODE Rescaling Lemma]\label{ODE_rescaling_lemma}
  Under the rescaling $(\alpha g,\sqrt{\alpha}F)$, the ODE for $\phi$ in Lemma \ref{poten_metric_ODE} becomes
  \begin{equation}\label{rescaled_phi_ODE}
      \phi''=-\frac{3}{2\alpha}(\phi^2)'+\left(\frac{1}{\alpha}\right)\phi^{3}-\frac{\lambda}{\alpha}\phi.
  \end{equation}
\end{lemma}
\begin{proof}
  Recall that $r$ is the radial arclength parameter with respect to the metric $g$. Letting $\tilde{r}$ be the radial arclength parameter with respect to $\alpha g$, we have that $\tilde{r}=\sqrt{\alpha}r$ and $\frac{\partial}{\partial \tilde{r}}=\frac{1}{\sqrt{\alpha}}\frac{\partial}{\partial r}$. Hence
  \[
    \frac{\partial}{\partial \tilde{r}}f(r) = \frac{1}{\sqrt{\alpha}}\frac{\partial}{\partial r} f = \frac{\mu}{\sqrt{\alpha}}\phi(r).
  \]

  To see how $\eta$ behaves under the rescaling $(g,F)\mapsto (\alpha g,\sqrt{\alpha}F)$, note that we have the following scaling laws:
  \begin{enumerate}
    \item $\Delta_{\alpha g}f=\frac{1}{\alpha}\Delta_{g}f$
    \item $dV_{\alpha g} = \alpha^{n/2} dV_{g} = \alpha dV_{g}$
  \end{enumerate}
  Hence
  \[
      \sqrt{\alpha} F = \sqrt{\alpha} \eta e^{-f}dV_{g} = \frac{\eta}{\sqrt{\alpha}} e^{-f}dV_{\alpha g},
  \]
  so the constant $\eta$ scales as $\frac{\eta}{\sqrt{\alpha}}$. Together these scaling laws imply that $\lambda$ scales as $\frac{\lambda}{\alpha}$. Now under the rescaling $(\alpha g,\sqrt{\alpha}F)$, the ODE for $\phi$ in Lemma \ref{poten_metric_ODE} becomes
  \begin{equation}
      \phi''=-\frac{3}{2\alpha}(\phi^2)'+\left(\frac{1}{\alpha}\right)\phi^{3}-\frac{\lambda}{\alpha}\phi.
  \end{equation}
\end{proof}

\begin{corollary}\label{sphere_solution_corollary}
  The round sphere solution appears as the limit of rescaled solutions $(\lambda g,\sqrt{\lambda}F)$ as $\lambda\to\infty$.
\end{corollary}
\begin{proof}
  Take a sequence $\{\lambda_i\}$ with $\lim_{i\to\infty}\lambda_{i}\to\infty$, which yield a sequence of ODEs and corresponding solutions $\phi_{i}$'s. Rescaling by $\lambda_{i}$ yields a sequence of ODEs
  \[
      \phi''=-\frac{3}{2}\frac{1}{\lambda_{i}}(\phi^2)'+\left(\frac{1}{\lambda_{i}}\right)\phi^{3}-\phi
  \]
  which converges to the ODE for the spherical solution in the limit. 
\end{proof}

\begin{remark}
  We have that one of $\mu$ and $\eta$ may be fixed in this ansatz via a rescaling. Since $\mu<0$ is required for $\lambda<0$ and can always be ensured by a choice of orientation (by the reflection $f(-r)$), it is natural to assume $\mu<0$ and then scale to fix $\mu=-1$. From this perspective, $\eta$ is parametrizing the family of solitons from Lemma \ref{poten_metric_ODE} via the constant $\lambda$. The interval $[-2,\infty)$ is the maximal interval for $\lambda$ under the scaling which fixes $\mu=-1$.
\end{remark}

\subsection{Distinguished Solutions}
In the interval $\lambda\in[-2,\infty)$, there are three distinguished solutions at $\lambda=-2$, 0 and as $\lambda\to\infty$.

\begin{proposition}
  Taking $\mu=-1$ and $\eta=0$, which corresponds to $\lambda=-2$ and $F=0$, yields the Hamilton cigar soliton with $\phi(r) = 2\tanh(r/2)$.
\end{proposition}

There is a critical threshold at $\lambda=0$ where the solutions to the surface gradient Ricci-Yang-Mills soliton equations transition from being complete noncompact manifolds ($\lambda<0$) to compact metric spaces ($\lambda>0$, which are orbifolds). At this threshold we have a distinguished complete noncompact manifold which develops a cusp at infinity.

\begin{proposition}
  Fixing $\mu=-1$ and $\eta=2$, which corresponds to $\lambda=0$, the warping factor $\phi(r)$ is explicitly given by
  \[
    \phi(r) = \frac{r}{r^{2}+1}.
  \]
  This defines a complete metric which develops a cusp at infinity.
\end{proposition}
\begin{proof}
  When $\lambda=0$ the ODE for $f$ becomes
  \[
    f'' = \tfrac{1}{2}(f')^{2}-e^{2f}
  \]
  which has the solution $f(r) = \ln\left(\tfrac{1}{r^{2}+2}\right)$ for the initial conditions $f(0)=f'(0)=0$ and $f''(0)=-1$. Since $f'=-\phi$, we have that 
  \[
    \phi(r) = \frac{r}{r^{2}+1}.
  \]

  To see that the warped product metric $\R_{\geq 0}\times_{\phi}S^{1}$ is complete, let $\gamma(t)$ be a radial geodesic starting at the origin in $M$. By the rotational symmetry of the metric at only this point the set of isometries of $(M,g)$ is isomorphic to $O(2)$, and hence any geodesic passing through the origin is a rotation or reflection of $\gamma(t)$. Since $\gamma(t)$ has unit speed parametrization in the radial direction and the metric exists for all $r$, $\gamma(t)$ has infinite length and thus $g$ is complete.
\end{proof}

\begin{figure}
  \centering
  \includegraphics[scale=0.6]{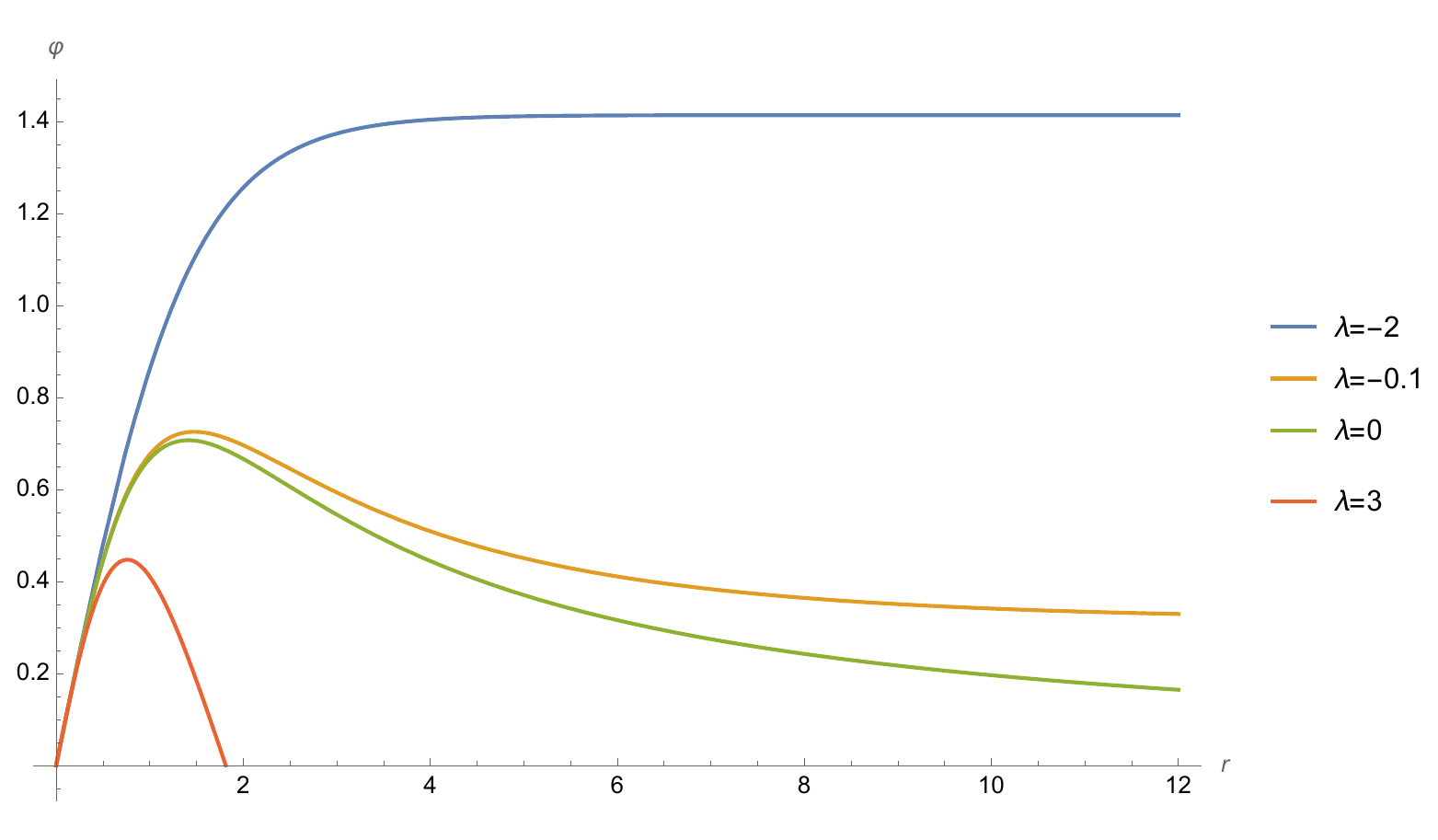}
  \caption{Numerical solutions for the profile function $\phi$ for different values of $\lambda$ assuming $\mu=-1$.}
  \label{fig:profilefunctions}
\end{figure}

We now prove existence for solutions in the remaining intervals for $\lambda$.

\subsection{Compact Case: \texorpdfstring{$\lambda>0$}{Lambda Positive}}
In this subsection we show that when $\lambda>0$ the solutions must be compact metric spaces.

\begin{proposition}
  If $\lambda>0$ then $f'$ returns to zero in finite time.
\end{proposition}
\begin{proof}
  For the sake of a contradiction, suppose $f'$ does not return to zero in finite time. Recall that
  \[
    f'' = \tfrac{1}{2}(f')^{2}-\tfrac{\eta^{2}}{4}e^{2f}+\tfrac{1}{2}\lambda.
  \]
  Since the initial conditions for $f$ are $f(0)=f'(0)=0$ and $f''(0)=-1$, we have that $f'<0$ and $f<0$ on $(0,t_{0})$ for some $t_{0}$. We first claim that $f'$ cannot be uniformly bounded from above. If this were the case, i.e. $f'<c$ for some $c<0$ on $(t_{0},\infty)$, then we can assume $f$ is sufficiently negative that 
  \[
    f'' = \tfrac{1}{2}(f')^{2}-\tfrac{\eta^{2}}{4}e^{2f}+\tfrac{1}{2}\lambda > \tfrac{1}{2}c^{2}-\tfrac{\eta^{2}}{4}e^{2f}+\tfrac{1}{2}\lambda>0
  \]
  since $\lambda>0$. This would imply that $f'$ must return to zero in finite time, a contradiction.

  Observe that if $f'$ is sufficiently close to zero, but $e^{2f}$ is still large enough that $\tfrac{1}{2}\lambda-\tfrac{\eta^{2}}{4}e^{2f}<\tfrac{1}{2}(f')^2$, then $f''<0$ and $f'$ will decrease. Since $f'$ cannot be uniformly bounded from above, there must be a point such that $f''>0$, and thus it follows that eventually $f$ decreases until we have
  \[
    \tfrac{1}{2}\lambda-\tfrac{\eta^{2}}{4}e^{2f}>\tfrac{1}{2}(f')^2 \iff f''>0
  \]
  for some $r_{0}$. Since at this point $f'<0$ by assumption, $f''$ will become increasingly positive for $r>r_{0}$, as $f'$ monotonically increases to zero while $f''>0$ and $f$ monotonically decreases while $f'<0$. Hence for $r>r_{1}>r_{0}$ we will have $f''>c(r_{0})$ for some $c>0$, and thus $f'$ reaches 0 in finite time, which is a contradiction.
\end{proof}

\begin{corollary}\label{compact_cor}
  The warping factor $\varphi$ returns to zero in finite time, and thus for $\lambda>0$ the surface is compact (but not necessarily a smooth manifold).
\end{corollary}
 
\begin{remark}
  In (\ref{compact_cor}), the surfaces may not be smooth manifolds since the warping factor $\varphi$ must satisfy the conditions $\varphi(L)=0$ and $\varphi'(L)=-1$ to extend smoothly to a metric on the sphere. In this case, $\varphi'(L)>-1$ if $\lambda>0$, so the radius where the warping factor returns to zero is a cone point. These spaces are ``football'' orbifolds.

  The fact that $\varphi'(L)>-1$ when $\lambda>0$ is due to smooth dependence of solutions to ODEs on initial conditions and parameters, and that $\varphi'(L)$ can be made arbitrarily close to 0 from below by taking $\lambda$ close to 0 from above. As $\lambda$ approaches $\infty$, $\varphi'(L)$ must decrease to -1, as the sphere solution in the limit satisfies $\varphi'(L)=-1$.
\end{remark}

We do not use the following fact here, but knowing that $\phi$ is periodic implies that there is a quantity involving $\phi$ that is conserved along the flow lines of the ODE from the soliton equations.

\begin{corollary}
  The function $\phi$ is periodic.
\end{corollary}
\begin{proof}
  This follows from $-\phi(-r)$ also being a solution to the ODE in \ref{poten_metric_ODE} and $\phi$ returning to zero in finite time.
\end{proof}

\subsection{Noncompact Case: \texorpdfstring{$-2\leq\lambda\leq0$}{Lambda Negative}}
In this section we prove that when $\lambda\in(-2,0)$, the solutions to the soliton equations are complete, noncompact manifolds. The edge cases of $\lambda=-2$ and $\lambda=0$ were discussed previously and shown to be noncompact solutions (the cigar solution and cusp solution).

\begin{proposition}\label{existence_proof}
  Let $f\in C^{\infty}(M)$ be the soliton potential function satisfying the ODE in Lemma \ref{poten_metric_ODE} where $\mu=-1$. Given initial conditions $f(0)=0$, $f'(0)=0$ and $f''(0)=-1$, if $\lambda<0$ we have $f'(r)$ decreases monotonically while $-\sqrt{-\lambda}<f'(r)<0$.
\end{proposition}
\begin{proof}
  We will first show that $f(r)$ and $f'(r)$ initially must be negative for at least some small $r$, and then $f'(r)$ must remain below zero while decreasing to $-\sqrt{-\lambda}$. First, we claim there exists $\epsilon>0$ such that $f'(r)<0$ on $(0,\epsilon)$. This follows by the continuity of $f''(r)$ and the initial condition that $f''(0)=-1<0$, so there exists $\epsilon>0$ such that $f''(r)<0$ on $(0,\epsilon)$ and integrating yields the desired result for $f'$. An analogous argument shows $f(r)<0$ on $(0,\epsilon)$.

  We next claim that if $-\sqrt{-\lambda}<f'(s)<0$, $f'$ monotonically decreases. From the ODE in Corollary \ref{f_2ndorder_eqt}, we have
  \[
    f'' = \frac{1}{2}\left(f'\right)^2-\frac{\eta^{2}}{4}e^{2f}+\frac{1}{2}\lambda.
  \]
  Assuming $\lambda<0$, then
  \[
    f'' < \frac{1}{2}\left(f'\right)^{2}+\frac{1}{2}\lambda<0
  \]
  which implies $f'$ is monotonically decreasing while $-\sqrt{-\lambda}<f'(r)<0$.
\end{proof}

We have not shown that $f'$ must reach below $-\sqrt{-\lambda}$, but the next lemma will be important for arguments were we do assume $f'$ crosses below this value.

\begin{lemma}\label{f'_upperbnd_lemma}
  If $f'(s)\leq -\sqrt{-\lambda}$ for some $s\in(0,\infty)$, then for $r>s$ we have $f'(r)<-\sqrt{-\lambda}$.
\end{lemma}
\begin{proof}
  Note that when $f'=-\sqrt{-\lambda}$, the ODE for $f''$ implies
  \[
    f'' = -\frac{\eta^{2}}{4}e^{2f}<0
  \]
  and thus $f'$ may only decrease through $-\sqrt{-\lambda}$. This implies that $f'$ may only cross this value once by continuity. 
\end{proof}

\begin{corollary}
  Given $\delta>0$, $f'$ is strictly bounded from above by a negative constant depending on $\delta$ for $r>\delta>0$.
\end{corollary}
\begin{proof}
  This follows from combining Lemma \ref{existence_proof} with Lemma \ref{f'_upperbnd_lemma}.
\end{proof}

\begin{lemma}\label{potent_unif_lwrbnd}
  The function $f'$ is uniformly bounded from below on $[0,\infty)$. Furthermore, if $f'$ has a minimum at $r>0$, this minimum must be unique, and $f''(s)\geq 0$ for $s>r$.
\end{lemma}
\begin{proof}
  Note that if $f'$ does not pass $-\sqrt{-\lambda}$, we are done automatically. Now suppose $f'$ does pass $-\sqrt{-\lambda}$. We first argue that there is an $r>0$ such that $f''(r)=0$. Suppose for the sake of contradiction that $f'$ is not bounded from below. Since $f'$ is strictly bounded from above by a negative constant for $r>\delta>0$, we can assume for any $\epsilon$ that we can take $r$ sufficiently large to have $\tfrac{\eta^{2}}{4}e^{2f}<\epsilon$. In particular, we note that since $f'$ is not bounded from below, we can take $f'$ sufficiently large that $f''>0$. Let $r_1$ denote one such point where $f''(r_1)>0$ and $f'(r_1)<-\sqrt{-\lambda}$.

  Again using that $f'$ is unbounded below, there exists $r_2$ such that $f'(r_2)<f'(r_1)$. In particular this means that $f''<0$ at some point in $(r_1,r_2)$, so applying the Intermediate Value Theorem there must be an $r_3\in (r_1,r_2)$ such that $f''(r_3)=0$. We differentiate the ODE for $f''$ to see
  \[
    f^{(3)}=f'(f'')-\tfrac{\eta^{2}}{4}e^{2f}(2f')
  \]
  which implies $f^{(3)}(r_3)>0$. This means that the point $r_3$ is a local minimum for $f'$, but we also see that for any point where $f''=0$ we must have $f^{(3)}>0$, so $f'$ only has local minima when $f'<0$ (and thus $f''\geq 0$ for $r>r_3$). This contradicts that $f'$ is not bounded from below, since we could otherwise find a local maxima by taking a further point $r_4$ where $f'(r_4)<f'(r_2)$ and noting $f'(r_4)<f'(r_3)<f'(r_2)$ while $r_3<r_2<r_4$. Thus $f'$ must be uniformly bounded from below on $[0,\infty)$, and the computation for $f^{(3)}$ shows that if $f'$ achieves a minimum then the minimum is unique.
\end{proof}

\begin{corollary}\label{potent_asympt_growth}
  The potential function $f$ has asymptotic linear growth given by
  \[
    \lim_{r\to\infty}f'(r) = -\sqrt{-\lambda}.
  \]
\end{corollary}
\begin{proof}
  There are two cases: $f'$ is below zero but stays above $-\sqrt{-\lambda}$ and $f'$ is below zero but crosses past $-\sqrt{-\lambda}$. For the first case, Lemma \ref{existence_proof} showed that $f'$ is monotonically decreasing while $-\sqrt{-\lambda}<f'<0$. Since we assume $f'>-\sqrt{-\lambda}$, we have $f$ is a monotonically decreasing function bounded from below, and thus $\lim_{r\to\infty} f'(r)=-\sqrt{-\lambda}$.

  For the second case, Lemma \ref{f'_upperbnd_lemma} showed that $f'$ is bounded above by $-\sqrt{-\lambda}$ once it crosses this value, and Lemma \ref{potent_unif_lwrbnd} showed that $f'$ must also be bounded from below and have a unique minimum. If $f'$ achieves this minimum, it is monotonically increasing after, and thus we have a monotonically increase function bounded from above by $-\sqrt{-\lambda}$, so $\lim_{r\to\infty} f'(r) = -\sqrt{-\lambda}$ again.

  If $f$ does not achieve a minimum, it will be monotonically decreasing. Since it is also bounded from below, the limit as $r\to\infty$ must exist, and we now apply
  \[
    \lim_{r\to\infty} f''(r) = \lim_{r\to\infty}\frac{1}{2}(f'(r))^{2}-\lim_{r\to\infty}\frac{\eta^{2}}{4}e^{2f}+\frac{1}{2}\lambda=\lim_{r\to\infty}\frac{1}{2}(f'(r))^{2}+\lambda
  \]
  where we used that $f'$ is bounded from above by a negative constant to imply $\lim_{r\to\infty}e^{2f}\to 0$. Note that $f$ is smooth and thus $f''$ is continuous, so $\lim_{r\to\infty}f'(r)=L$ implies $\lim_{r\to\infty}f''(r)=0$. Substituting to the above limit of the ODE and solving for $f'$ yields $\lim_{r\to\infty}f'(r)=-\sqrt{-\lambda}$ in this final case.
\end{proof}

\begin{corollary}\label{COR_unif_bounds}
  The potential function $f$ satisfies that for $r>\delta>0$, $-C_2<f'<-C_1$ for $C_2>C_1>0$ depending on $\delta$, and is thus defined for all $r$. Furthermore, since $f'=\mu\varphi$, $\varphi$ is defined for all $r$ and $\varphi>-\mu C_1$ for $r>\delta>0$.
\end{corollary}

\begin{theorem}
  There exists a one-parameter family of rotationally symmetric, complete, noncompact gradient Ricci-Yang-Mills solitons in two dimensions. These solitons are asymptotic to cylinders and are parametrized by $\lambda\in(-2,0)$. 
\end{theorem}
\begin{proof}
  Recall that $f'=\mu\varphi$ for a gradient Ricci-Yang-Mills soliton by Lemma \ref{YMSurfSolMetricEq}, where we have rescaled to fix $\mu=-1$. The initial conditions on $\varphi$ are known to be $\varphi(0)=0$, $\varphi'(0)=1$, and by setting $f(0)=0$ we satisfy the initial conditions for the setup of Proposition \ref{existence_proof}. Since $\varphi$ is uniformly bounded away from 0 away from the origin by Corollary \ref{COR_unif_bounds}, the warped product metric on $\R_{\geq 0}\times_{\varphi}S^{1}$ is complete. By Lemma \ref{potent_asympt_growth}, we know that $\lim_{r\to\infty}f'(r)=-\sqrt{-\lambda}$. This means $\varphi$ is asymptotically constant and the metrics are asymptotic to cylindrical metrics.
\end{proof}

The following computation shows that the solitons constructed in the preceding theorem are pairwise non-isometric.
\begin{proposition}
  The curvature of the surface at the origin is given by
  \[
    \lim_{r\to 0^{+}}K(r) = \lim_{r\to 0^{+}}\Ric(\grad r,\grad r) = \lim_{r\to 0^{+}}-\frac{\phi''}{\phi} = -(\mu-\tfrac{\eta^{2}}{2}).
  \]
\end{proposition}
\begin{proof}
  Since we are considering surfaces, there is only one sectional curvature and it is equivalent to the Ricci curvature. By Lemma \ref{YMSurfSolMetricEq} we know that in the radial direction the Ricci curvature is given by $-\tfrac{\phi''}{\phi}$, and by Lemma \ref{poten_metric_ODE} we see that $\lim_{r\to 0^+}\phi''(r) = 0$. Applying L'Hopital's rule, we have
  \[
    \lim_{r\to0^{+}}K(r) = \lim_{r\to 0^{+}}-\frac{\phi^{(3)}}{\phi'} = \lim_{r\to 0^{+}}-\frac{f^{(4)}}{f''}.
  \]
  Again by Lemma \ref{poten_metric_ODE}, recall that
  \begin{align*}
    f^{(3)} &= \tfrac{1}{2}\left((f')^{2}\right)'-\tfrac{\eta^{2}}{4}(e^{2f})'\\
      &= f'f''-\tfrac{\eta^{2}}{2}f'e^{2f}
  \end{align*}
  which we differentiate and use to replace the $f'''$ term to get
  \begin{align*}
    f^{(4)} &= (f'')^2+f'(f'f''-\tfrac{\eta^{2}}{2}f'e^{2f})-\tfrac{\eta^{2}}{2}(f''e^{2f}+2(f')^{2}e^{2f}).
  \end{align*}
  Finally, since $f(0)=f'(0)=0$ and $f''(0)=\mu$, 
  \[
    \lim_{r\to 0^{+}}-\frac{f^{(4)}}{f''} = -\frac{\mu(\mu-\tfrac{\eta^{2}}{2})}{\mu} = -\left(\mu-\tfrac{\eta^{2}}{2}\right).
  \]
\end{proof}

\section{Proof of Theorem 1.1}
We continue to fix $\mu=-1$ so that $\lambda\in[-2,\infty)$. Recall that are three special cases for the family of solitons; $\lambda = -2$ corresponds to the classical Hamilton cigar soliton, $\lambda = 0$ corresponds to a soliton developing a cusp at infinity, and as $\lambda$ increases to $\infty$ the solitons approach a round point.

\begin{theorem}\label{convergencethm}
The rotationally symmetric steady gradient Ricci-Yang-Mills solitons on surfaces form a 1-parameter family, where $\lambda\in [-2,\infty)$ parametrizes the family. We have the following: The family converges
  \begin{enumerate}[label=(\arabic*)]
    \item in the pointed Cheeger-Gromov sense to the Hamilton cigar soliton as $\lambda\to -2^{+}$;
    \item in the pointed Cheeger-Gromov sense to the cusp soliton as $\lambda\to 0^{-}$;
    \item in the pointed Gromov-Hausdorff sense to the cusp soliton as $\lambda\to 0^{+}$;
    \item in the Gromov-Hausdorff sense, after the rescaling discussed for \eqref{rescaled_phi_ODE}, to $S^{2}$ with the round metric as $\lambda\to \infty$.
  \end{enumerate}
\end{theorem}
\begin{proof}
  The convergence in each case follows by applying smooth dependence of solutions to ODEs with respect to initial conditions and coefficients in the ODE (see Remark 1, page 92 of \cite{perkoDiffEqtsDynamSys2001}). In each case here, the initial conditions are fixed by the conditions required for $\phi$ to extend smoothly to a metric on the warped product, but we vary the parameter $\lambda$ of the ODE for $\phi$ in Lemma \ref{poten_metric_ODE}.

  Observe that to prove (pointed)-Gromov-Hausdorff convergence of the spaces here, it is sufficient to prove that on compact intervals $[0,R]$ the warping factors $\phi_i$ converge uniformly to the target warping factor $\phi$. This is because if the $\phi_i$ are uniformly close to $\phi$, the diameters of $(M,g_i)$ will be uniformly close to $(M,g)$ where $g_i=dr^2+\phi_i^{2}d\theta^2$ and $g=dr^2+\phi^2d\theta^2$, and thus $(M,g_i)$ can be made arbitrarily close to $(M,g)$ in the Gromov-Hausdorff topology.

  For $\lambda\to -2^{+}$ and $\lambda\to 0$ with fixed $R>0$, we have uniform convergence of the solutions $\phi_{\lambda}$ on $r\in[0,R]$ by \cite{perkoDiffEqtsDynamSys2001}, which implies smooth convergence at the level of the metric. Thus the underlying spaces converge in the Cheeger-Gromov sense, yielding (1) and (2).

  For $\lambda\to\infty$, we consider the rescaling of the ODE for the warping function $\varphi$ from \eqref{rescaled_phi_ODE}. Corollary \ref{sphere_solution_corollary} showed that the rescaled ODEs for $\varphi$ limit to the ODE for the round sphere solution. Convergence of the solutions follows by another application of smooth dependence of ODEs on initial data and ODE parameters from \cite{perkoDiffEqtsDynamSys2001}, which implies the spaces converge in the Gromov-Hausdorff sense. Note that this only holds in the Gromov-Hausdorff sense since the solutions correspond to orbifolds converging to the round sphere.
\end{proof}

\section{Proof of Classification}
Our construction of the 1-parameter family of rotationally symmetric solitons was based on the ansatz that $f$ had a critical point. In this section we seek to show that any solution to the ODE for $f$ from \ref{f_2ndorder_eqt} necessarily has a critical point. Note that the ODE analysis here relies on the assumption that the bundle curvature is nonzero, since the bundle curvature is providing the $e^{2f}$ term in the ODE. We first demonstrate that the potential function $f$ for any Ricci-Yang-Mills soliton on a surface satisfies this ODE locally.

\begin{proposition}
  The equation for $f$ in (\ref{poten_metric_ODE}) holds where $\frac{\partial}{\partial r} = \frac{\nabla f}{|\nabla f|}$.
\end{proposition}
\begin{proof}
  Assume that we are not at a fixed point for $f$, which in particular implies that $J\nabla f$ is a nontrivial Killing field on a local coordinate patch. Since $J\nabla f$ and $\tfrac{\nabla f}{|\nabla f|}$ commute, by the Frobenius theorem we have local coordinates in the $\tfrac{\nabla f}{|\nabla f|}$ (denote by the coordinate $r$) direction and $J\nabla f$ (denote by the coordinate $\theta$) direction. The metric can thus be written as the warped product $g = dr^{2}+\phi^{2}(r,\theta)d\theta^{2}$. Since $J\nabla f$ is Killing, its norm is constant along its integral curves, which implies that $\phi$ is independent of $\theta$, and we have
  \[
    g=dr^{2}+\phi^{2}(r)d\theta^{2}.
  \]
  By interpreting $'$ to mean derivatives in the $\tfrac{\nabla f}{|\nabla f|}$ direction, we derive the system (\ref{poten_metric_ODE}) in this coordinate patch without requiring the existence of a fixed point.
\end{proof}

\begin{remark}
  If we assume that we have a complete steady RYM soliton, then the above construction extends beyond a local coordinate patch. The definition of $r$ comes from the geodesic flow of $\nabla f/|\nabla f|$, and by the completeness assumption this orbit must exist for all time. In the following theorems this fact justifies taking limits as $r\to\infty$. 
\end{remark}

Now the problem of classifying the steady gradient RYM solitons on surfaces can be reduced to showing that given a complete solution, the potential function $f$ must have a critical point, which then implies it is rotationally symmetric. Thus the soliton falls into one of the categories of Theorem \ref{convergencethm}. We now prove that the $f$ ODE from (\ref{poten_metric_ODE})
\[
    f'' = B-Ae^{2f}+\frac{1}{2}(f')^{2}
\]
must have a critical point for $f$ in the three cases of $B>0$, $B=0$ and $B<0$.

\begin{lemma}[$B>0$ Case]\label{f_crit_point_pos_lemma}
  Let $f$ be a solution to
  \[
    f'' = B-Ae^{2f}+\frac{1}{2}(f')^{2}
  \]
  where $B>0$. Then $f$ must have a critical point.
\end{lemma}
\begin{proof}
  Without loss of generality, we can assume that we specify initial conditions at $r=0$, i.e. $f(0)=c_1$ and $f'(0)=c_2$. Note that $f(-r)$ also solves the equation with initial conditions $f(0)=c_0$ and $f'(0)=-c_1$, which means we can reflect solutions to flip the sign of the derivative. 

  When $f>\tfrac{1}{2}\ln(\tfrac{B}{A}):=f_{th}$ this implies $B-Ae^{2f}<0$ (since $B-Ae^{2f_{th}}=0$). We will refer to $f_{th}$ as the \textit{threshold}. If for some $r\in\R$ we have $f(r)<f_{th}$ and $f'(r)<0$, this implies $f''(s)$ is uniformly bounded from below by a positive constant for $s>r$ until $f$ hits a critical point. The condition that $f'(r)<0$ can always be achieved by reflecting $f$ if necessary. Thus in this case, $f'$ must reach zero in finite time, and we are guaranteed a critical point in this case.

  We now consider the case that $f(0)>f_{th}$ and $f''(0)<0$. Reflecting the solution does not change $f''(0)$, so we can take $f'(0)<0$. In this case $f$ will decrease as $r$ increases, and if $f''\leq 0$ is preserved then $f$ must decrease below the threshold eventually which would imply $f$ is now in the first situation above. In fact, the only obstruction to reaching a critical point must be the case where $f(r)>f_{th}$, $f'(r)<0$ for all $r>0$ with $\lim_{r\to\infty}f(r)=f_{th}$ and $\lim_{r\to\infty}f'(r)=0$ (which implies $\lim_{r\to\infty}f''(r)=0$). If $f''$ is positive for too long, then $f$ clearly hits a critical point, and if $f''$ possibly oscillates about 0 (but does not limit to zero), either $f$ hits a critical point or $f$ decreases below the threshold (which then implies $f$ will hit a critical point).

  Thus it is sufficient to prove that in the setting of $f(r)>f_{th}$, $f'(r)<0$ for all $r$, $\lim_{r\to\infty}f(r)=f_{th}$ and $\lim_{r\to\infty}f'(r)=\lim_{r\to\infty}f''(r)=0$ that $f$ must reach a critical point. Suppose now for the sake of contradiction that we are in this case but $f$ does not have a critical point. By assumption there exists a sequence $\{r_{i}\}$ such that $\lim_{i\to\infty}f''(r_{i})=0$ and $\lim_{i\to\infty}f'(r_{i})=0$. Let $\epsilon>0$, and we can find some $r_{k}$ such that $0<f''(r_{k})<\epsilon$ and $-\epsilon<f'(r_{k})<0$. 

  We start by using the setup to determine stricter bounds on $f$ and $f'$ from the ODE in this setting. Since $f''(r_{k})>0$ we have that
  \[
    -2(B-Ae^{2f})<(f')^{2} \implies f'<-\sqrt{-2(B-Ae^{2f})}.
  \]
  Applying the bound on $f'$ yields a bound on $f$: 
  \[
    -\epsilon<f'<0 \implies -\epsilon<-\sqrt{-2(B-Ae^{2f})}
  \]
  which we solve for $f$ to yield
  \[
    f<\tfrac{1}{2}\ln\left(\tfrac{1}{A}\left(\tfrac{\epsilon^{2}}{2}+B\right)\right).
  \]
  Note that these inequalities rely on $f''>0$. Since $f^{(3)}=f'(f''-2Ae^{2f})$, the assumption that $f'<0$ along with the bounds at $r_{k}$ yield that $f^{(3)}(r_k)>0$. In order for $f^{(3)}$ to become negative (and before $f$ has a critical point), it must be that $f''-2Ae^{2f}>0$, and in particular this requires $f''>2B$ (since $f>f_{th}$). This is not possible under the assumption that $|f'|<\epsilon$ for $\epsilon$ sufficiently small, so $f''>0$ must hold in the regime we consider here.

  The goal now is to show $f'$ is sufficiently large that it forces $f$ below the threshold in finite time. From the bounds above, the distance from $f$ to the threshold value is less than
  \[
    \tfrac{1}{2}\ln\left(\tfrac{1}{A}\left(\tfrac{\epsilon^{2}}{2}+B\right)\right) - \tfrac{1}{2}\ln\left(\tfrac{B}{A}\right) = \tfrac{1}{2}\ln\left(\tfrac{\epsilon^{2}}{2B}+1\right) := \delta(\epsilon, B)
  \]
  so if $f$ does not move below the threshold, it must be that
  \[
    f(r)-f(r_k) = \int_{r_{k}}^{r}f'(t)dt > -\delta
  \]
  for all $r>r_{k}$. Break the interval $[f_{th},f_{th}+\delta]$ into subintervals $[f_{th}+\tfrac{\delta}{2^i}, f_{th}+\tfrac{\delta}{2^{{i-1}}}]$. The bound on $f'$ can be used to bound the time it takes to drop $f$ through these subintervals by a convergent geometric series, which implies $f$ must go below the threshold in finite time. Define $r_{i}$ to be the time that $f(r_{i})=f_{th}+\tfrac{\delta}{2^i}$. Recalling that $f''>0$, we have that on the $i$-th subinterval,
  \begin{align*}
    -\frac{\delta}{2^{i}} = \int_{r_{i-1}}^{r_{i}}f'(t)dt &< \int_{r_{i-1}}^{r_{i}}-\sqrt{-2(B-Ae^{2f})}\;dt\\
      &< (r_{i}-r_{i-1})\left(-\sqrt{-2(B-Ae^{2(f_{th}+\delta/2^{i})})}\right)\\
      &= (r_{i}-r_{i-1})\left(-\sqrt{-2B(1-e^{\delta/2^{i}})}\right)   \tag*{(using $Ae^{2f_{th}}=B$)}
  \end{align*}
  which implies 
  \[
    r_{i}-r_{i-1}< \frac{\delta}{2^{i}}\frac{1}{\sqrt{2B}}\frac{1}{\sqrt{e^{\delta/2^{i}}-1}}<\frac{\delta}{2^{i}}\frac{1}{\sqrt{2B}}\frac{1}{\sqrt{\delta/2^{i}}}=\sqrt{\frac{\delta}{2B}}\left(\frac{1}{\sqrt{2}}\right)^{i}.
  \]
  Summing these and noting that the right side is the desired convergent geometric series implies that there is a sequence of points $\{r_{i}\}$ such that $\lim_{i\to\infty}f(r_{i})=f_{th}$ and $\lim_{i\to\infty}r_{i}=R<\infty$. This contradicts the assumption that $f$ remains above the threshold. Since negating any of the assumptions that set up the contradiction implies the existence of a critical point, we are done.
\end{proof}

\begin{lemma}[$B=0$ Case]\label{f_crit_point_zero_lemma}
  Let $f$ be a solution to
  \[
    f'' = \frac{1}{2}(f')^{2}-Ae^{2f}.
  \]
  Then $f$ must have a critical point.
\end{lemma}
\begin{proof}
  This ODE has explicit solutions for initial value problems $f(0)=c_1$, $f'(0)=c_2$ given by
  \begin{align*}
    f(x) &= \ln\left(\frac{4e^{c_{1}}}{(2-c_{2}x)^{2}+2Ae^{2c_{1}}x^{2}}\right)\\
    \implies f'(x) &= \frac{4c_{2}-2(c_{2}^{2}+2Ae^{2c_{1}})x}{(2-c_{2}x)^{2}+2Ae^{2c_{1}}x^{2}}.
  \end{align*}
  Since the denominator is nonnegative for all $x\in\R$, $f$ clearly must have a critical point.
\end{proof}

\begin{lemma}[$B<0$ Case]\label{f_crit_point_neg_lemma}
  Let $f$ be a solution to
  \[
    f'' = \frac{1}{2}(f')^{2}-(B+Ae^{2f})
  \]
  where $B>0$. Then $f$ must have a critical point.
\end{lemma}
\begin{proof}
  Recall that in Lemma \ref{potent_asympt_growth} we showed that if $f'(0)<0$, $f'$ limits to $-\sqrt{2B}$ as $r\to\infty$. By reflecting, we can now consider $f'(0)>0$. Note that since $f''<\tfrac{1}{2}(f')^{2}-B$, if $\tfrac{1}{2}(f')^{2}<B$ then $f''<0$ and $f'$ decreases. In this scenario, this implies $f''$ is bounded above by a negative constant, and thus $f'$ is forced to 0 in finite time. 

  If we suppose $f'>0$ for all $r>0$, then we must have that $\tfrac{1}{2}(f')^2\geq B$ for $r>0$. But this implies that $f'\geq \sqrt{2B}$ for $r>0$, and by Lemma \ref{potent_asympt_growth} we would have $\lim_{r\to-\infty}f'(r)=-\sqrt{2B}$. Thus by the Intermediate Value Theorem $f'$ must have a fixed point for some finite $r$, and we have shown in all cases that $f$ must have a critical point.
\end{proof}

Collecting the above lemmas yields a classification of the solitons on surfaces.
\begin{theorem}\label{classification_thm}
  Any complete Ricci-Yang-Mills soliton on a surface is isometric to a member of the family in Theorem \ref{convergencethm} up to rescaling.
\end{theorem}
\begin{proof}
  Collecting Lemmas \ref{f_crit_point_pos_lemma}, \ref{f_crit_point_zero_lemma} and \ref{f_crit_point_neg_lemma} yield that the potential function of any Ricci-Yang-Mills soliton must have a critical point. By Lemma \ref{fix_point_rot_sym_lemma} any complete surface Ricci-Yang-Mills soliton must be rotationally symmetric. Therefore, it must be a member of the family described in Theorem \ref{convergencethm} up to rescaling.
\end{proof}

\printbibliography

@book{perkoDiffEqtsDynamSys2001,
	author = {Perko, Lawrence},
	edition = {3},
	publisher = {Springer},
	series = {Texts in Applied Mathematics},
	title = {Differential Equations and Dynamical Systems},
	year = {2001}}

@book{mgf_streetsGRF2021,
	author = {Mario Garcia-Fernandez and Jeff Streets},
	edition = {1},
	publisher = {AMS},
	series = {University Lecture Series},
	title = {Generalized Ricci Flow},
	volume = {76},
	year = {2021}}

@book{chow_etalRicciFlowVol1_2007,
	author = {Bennet Chow and Sun-Chin Chu and David Glickenstein and Christine Guenther and James Isenberg and Tom Ivey and Dan Knopf and Peng Lu and Feng Luo and Lei Ni},
	publisher = {AMS},
	series = {Mathematical Surveys and Monographs},
	title = {The Ricci Flow: Techniques and Applications: Part I: Geometric Aspects},
	volume = {135},
	year = {2007}}

@book{chowLuNi_HamiltonsRicciFlow,
	author = {Bennet Chow and Peng Lu and Lei Ni},
	publisher = {AMS},
	series = {Graduate Studies in Mathematics},
	title = {Hamilton's Ricci Flow},
	volume = {77},
	year = {2006}}

@book{chow_RicciSolLowDim2023,
	author = {Bennet Chow},
	publisher = {AMS},
	series = {Graduate Studies in Mathematics},
	title = {Ricci Solitons in Low Dimensions},
	volume = {235},
	year = {2023}}

@misc{bryantRFSol3DRotSym2005,
	author = {Robert Bryant},
	title = {Ricci flow solitons in dimension three with SO(3)-symmetries},
	year = {2005},
	publisher = {Unpublished manuscript},
	url={https://sites.math.duke.edu/~bryant/3DRotSymRicciSolitons.pdf}}

@article{garciafernandezMolinaStreetsPluriclosedFlowHullStrominger2026,
    title = {Pluriclosed flow and the Hull-Strominger system},
	journal = {Advances in Mathematics},
	volume = {485},
	pages = {110699},
	year = {2026},
	issn = {0001-8708},
	doi = {https://doi.org/10.1016/j.aim.2025.110699},
	author = {Mario Garcia-Fernandez and Raul {Gonzalez Molina} and Jeffrey Streets}}

@article{streetsClassSolitonsPluriclosedFlow2019, 
      title={Classification of solitons for pluriclosed flow on complex surfaces}, 
      volume={375}, 
      DOI={10.1007/s00208-019-01887-4}, number={3–4}, 
      journal={Mathematische Annalen}, 
      author={Streets, Jeffrey}, 
      year={2019}, 
      month={8}, 
      pages={1555–1595}}

@phdthesis{streetsRYMFlow2007,
	author={Jeffrey Streets},
	title={Ricci Yang-Mills flow},
	year={2007},
	school={Duke University}}

@article{podestaRafferoPosCurvSO3GenRS2025,
    title={Three-dimensional positively curved generalized Ricci solitons with SO(3)-symmetries}, 
    author={Fabio Podestà and Alberto Raffero},
	journal={Advances in Mathematics},
    year={2025},
	volume={479},
	pages = {110426},
	issn = {0001-8708}}

@article{chen_NoteUniformRiemSurf2006,
      title={A note on uniformization of Riemann surfaces by Ricci flow}, 
      author={Xiuxiong Chen and Peng Lu and Gang Tian},
      year={2006},
	  journal={Proceedings of the American Mathematical Society},
	  volume={134},
	  number={11},
	  pages={3391-3393}}

@article{streets_RicYMFlowSurfaces2009,
	title = {Ricci Yang–Mills flow on surfaces},
	journal = {Advances in Mathematics},
	volume = {223},
	number = {2},
	pages = {454-475},
	year = {2010},
	issn = {0001-8708},
	doi = {https://doi.org/10.1016/j.aim.2009.08.014},
	author = {Jeffrey Streets}}

@article{streetsUstinovskiy_ClassificationGenKRSolitonCmplxSurf2020,
	author = {Streets, Jeffrey and Ustinovskiy, Yury},
	title = {Classification of Generalized Kähler-Ricci Solitons on Complex Surfaces},
	journal = {Communications on Pure and Applied Mathematics},
	volume = {74},
	number = {9},
	pages = {1896-1914},
	doi = {https://doi.org/10.1002/cpa.21947},
	year = {2021}}

@book{polchinski_StringTheoryVol11998,
    author = {Polchinski, J.},
    title = {String theory. Vol. 1: An introduction to the bosonic string},
    doi = {10.1017/CBO9780511816079},
    isbn = {9780521633031},
    publisher = {Cambridge University Press},
    series = {Cambridge Monographs on Mathematical Physics},
    year = {1998}}
\end{document}